\documentclass[11pt,a4paper,twoside]{article}
\usepackage{amssymb,amsthm}
\usepackage{amsmath}
\usepackage{setspace}
\usepackage[a4paper,top=33mm, bottom=27mm, left=20mm, right=20mm]{geometry}
\usepackage[small, margin=20pt]{caption}
\usepackage{url}
\usepackage{color}

\usepackage{hyperref}

\allowdisplaybreaks[4]

\usepackage{fancyhdr}
\allowdisplaybreaks[4]

\title{Existence of spanning $\F$-free subgraphs with large minimum degree}

\author{Guillem Perarnau\footnote{\noindent School of Mathematics, University of Birmingham, Montreal, Quebec, Canada. g.perarnau@bham.ac.uk.} \and Bruce Reed\footnote{\noindent  Kawarabayashi Large Graph Project,
National Institute of Informatics, Japan. breed@mcgill.ca.}}
\date{\today}

\theoremstyle{plain}
\newtheorem{theorem}{Theorem}
\newtheorem{lemma}[theorem]{Lemma}
\newtheorem{claim}[theorem]{Claim}

\theoremstyle{definition}
\newtheorem{conjecture}[theorem]{Conjecture}

\theoremstyle{definition}

\newcommand{\rst}[1]{\ensuremath{{\mathbin\upharpoonright}%
\raise-.5ex\hbox{$#1$}}} 

\renewenvironment{proof}[1][Proof]{\begin{trivlist}
\item[\hskip \labelsep {\textit{#1}.}]}{\qed\end{trivlist}}

\newcommand{\bP}{\Pr}
\newcommand{\bE}{\mathbb{E}}

\newcommand{\G}{J}
\newcommand{\F}{\mathcal{F}}

\newcommand{\GN}{\ell}
\newcommand{\Free}{\textbf{Free}}
\newcommand{\Bad}{\textbf{Bad}}

\newcommand{\eps}{\varepsilon}

\renewcommand{\a}{\alpha}

\newcommand{\ex}{\mbox{ex}}
\renewcommand{\hom}{\mbox{hom}}

\begin{document}

\pagenumbering{arabic}

\setcounter{section}{0}

\maketitle

\onehalfspace

\begin{abstract}
Let $\mathcal{F}$ be a family of graphs and let $d$ be large enough. For every 
$d$-regular graph $G$, we study the existence of a spanning $\mathcal{F}$-free subgraph 
of $G$ with large minimum degree. 
This problem is well-understood if $\mathcal{F}$ does 
not contain bipartite graphs. Here we provide asymptotically tight results for 
many families of bipartite graphs such as cycles or complete bipartite graphs.
To prove these results, we study a locally injective analogue of the question.
\end{abstract}

\section{Introduction}

Let $G=(V,E)$ be a $d$-regular graph on $n$ vertices. It is well-known that 
every $d$-regular graph has a spanning triangle-free
subgraph $H$ with minimum degree at least $d/2$.
Let $\F$ denote a family of graphs. We say that $G$ is $\F$-free if for every $F\in\F$, $G$ does not 
contain any subgraph isomorphic to $F$. In this paper we study the which is the largest minimum degree of a spanning $\F$-free subgraph $H$ of a $d$-regular graph $G$. 

Let $\ex(n,\F)$ be the maximum number of edges in an $\F$-free graph on $n$ 
vertices. Since the complete graph $K_n$ on $n=d+1$ vertices is 
$d$-regular, the minimum degree of a spanning $\F$-free subgraph $H$ of 
$K_{d+1}$ is at most $2\ex(d+1,\F)/(d+1)$. In~\cite{fkp2014}, Foucaud, Krivelevich 
and Perarnau conjectured the following.

\begin{conjecture}[\cite{fkp2014}]\label{conj:fkp}
For every family $\F$ there exists a constant $c_\F>0$ such that for every $d\geq 1$ and every $d$-regular graph $G$, there exists a spanning $\F$-free subgraph 
$H$ of $G$ with minimum degree $\delta(H)\geq c_\F\cdot \frac{\ex(d,\F)}{d}$.
\end{conjecture}

Here and throughout the paper, for any two functions $f,g$ that depend on several variables, including $d$, we will use the standard asymptotic notation $g=\Omega(f)$ to denote the following: given that all variables apart from $d$ are fixed, we have $\liminf_{d\to \infty} (g/f)>0$. We also use $g=O(f)$ to denote $f=\Omega(g)$, and $g=\Theta(f)$ if both $g=\Omega(f)$ and $g=O(f)$ hold. For instance, Conjecture~\ref{conj:fkp} can be written as:  given a family $\F$, for every $d$ and every $d$-regular graph $G$, there exists a spanning $\F$-free subgraph $H$ of $G$ with minimum degree $\delta(H)= \Omega(\ex(d,\F)/d)$.

If the chromatic number of $\F$\footnote{The chromatic number of  $\F$ is the smallest chromatic number of $F\in \F$.} is at least $3$~(i.e., there are no bipartite graphs in $\F$)
it is easy to verify the conjecture (see Proposition~1 in~\cite{fkp2014} for a precise result). 

The family $\F_g=\{C_3,\dots, C_{g-1}\}$ is of special interest since a graph $H$ is $\F_g$-free if and only if $H$ has girth at least $g$. 
Even if the order of $\ex(d,\F_g)$ is not known in general, there are some partial results towards Conjecture~\ref{conj:fkp} for $\F_g$. 
Kun~\cite{k2013} showed that for every $g\geq 4$, every $d$-regular graph $G$ admits a spanning $\F_g$-free
subgraph with minimum degree $\Omega(d^{1/g})$. 
Foucaud, Krivelevich and Perarnau~\cite{fkp2014} improved the lower bound on the minimum degree to $\Omega\left(\frac{\ex(d,\F_g)}{d\log{d}}\right)$. 
This shows that Conjecture~\ref{conj:fkp} holds up to a 
logarithmic factor for $\F_g$. 

In this paper we prove Conjecture~\ref{conj:fkp} for a large number of 
families $\F$ with chromatic number $2$, which in particular includes $\F_g$.

Before stating our main result we need to introduce some definitions. For any two graphs $F$ and $\G$, we say that $\varphi: V(F)\to V(\G)$ is an 
\emph{homomorphism} if for every $uv\in E(F)$ we have $\varphi(u)\varphi(v)\in E(\G)$. We say that $\varphi$ is \emph{locally injective} if, for every $v\in F$, the 
restriction of $\varphi$ onto the neighbors of $v$ in $F$ is injective. In other words, if $u_1$ and $u_2$ are neighbors of $v$ in $F$, then $\varphi(u_1)\neq\varphi(u_2)$. Let $\hom^*(F,\G)$ be the number of locally injective homomorphisms from $F$ to 
$\G$. Observe that the condition $\hom^*(F,\G)=0$, for every $F\in \F$, is 
stronger than $\F$-freeness. Any copy of $F$ in $\G$ induces an injective 
homomorphism from $F$ to $\G$ which, in particular, is also locally injective. 

For every graph $\G$, we denote by $\delta(\G)$ its minimum degree and by $\Delta(\G)$ its maximum degree. For every $\beta\geq 1$, we say that $\G$ is \emph{$\beta$-almost regular} if $\Delta(\G)\leq \beta\delta(\G)$.

Our first result is a locally injective version of Conjecture~\ref{conj:fkp}.
\begin{theorem}\label{thm:more_general}
Let $\F$ be a family of graphs, $\epsilon>0$, $\beta\geq 1$ and $\a \geq 25\beta^2$, and let $d$ be large enough. If $\G$ is a graph on $\a d$ vertices such that,
\begin{itemize}
\item[-] $\G$ is $\beta$-almost regular,
\item[-] $\G$ has minimum degree $\delta(\G)>d^{\eps}$, and
\item[-] $\hom^*(F,\G)=0$, for every $F\in \F$,
\end{itemize}
then, for every $d$-regular graph $G$ there exists a spanning subgraph $H$ of $G$ with $\delta(H)= 
\Omega(\delta(\G))$  and  $\hom^*(F,H)=0$, for every $F\in \F$.
\end{theorem}

The proof goes as follows. Initially, we select a bipartite subgraph $G'$ of $G$ with stable sets $A$ and $B$ and large minimum degree. 
The main part of the proof consists on finding a vertex coloring $\chi$ using as a color palette $V(\G)$. 
We color $A$ in a two steps procedure; the first one being random and the second one, deterministic. 
At the same time we delete some edges of $G'$ to obtain a subgraph $H'$ that has large minimum degree and where vertices in $B$ have rainbowly colored neighborhoods. 
We color $B$ afterwards via an iterative coloring procedure which also finishes with a deterministic step. 
Again, during the coloring procedure we delete some edges to obtain a subgraph $H$ that has minimum degree $\Omega(\delta(\G))$ and where all the neighborhoods are rainbowly colored. 
Moreover, the coloring satisfies that $\chi(a)\chi(b)$ is an edge of $\G$ for every edge $ab$ of $H$.
Thus, $\chi$ naturally provides a locally injective homomorphism from $H$ to $\G$. Since $\G$ has no locally injective copies of $F$, neither does $H$.

Theorem~\ref{thm:more_general} has some direct implications for Conjecture~\ref{conj:fkp}.
A family of graphs $\F$ is \emph{closed} if for every graph $\G$ we 
have: $\hom^*(F,\G)=0$ for every $F\in\F$ if and only if $\G$ is $\F$-free; that is, if $\G$ is $\F$-free, then there are no locally injective copies of $F\in\F$ in $\G$.
As a consequence of Theorem~\ref{thm:more_general} we obtain that Conjecture~\ref{conj:fkp} holds for every closed family.
\begin{theorem}\label{thm:main}
Conjecture~\ref{conj:fkp} holds for every closed family $\F$ of graphs: every $d$-regular graph $G$ has a spanning $\F$-free subgraph $H$ with minimum degree $\delta(H)= \Omega(\ex(d,\F)/d)$.
\end{theorem}

Theorem~\ref{thm:main} gives a lower bound on the minimum degree in terms of $\ex(d,\F)$. 
We now provide some explicit results for some important families $\F$ with chromatic number $2$ that are closed. 
The problem of determining $\ex(d,\F)$ when $\F$ contains 
bipartite graphs is one of the most important in extremal graph theory~(see~\cite{fs2013} for a 
complete survey on the topic). We use some well-known constructions of extremal graphs to provide explicit corollaries of Theorem~\ref{thm:main}:
\begin{itemize}
\item[-] Let $\F=\F_g=\{C_3,\dots, C_{g-1}\}$. Observe that $\F$ is a closed family. 
However, the asymptotic order of $\ex(d,\F)$ is not known in general. Using the 
Erd\H os-R\'enyi random graph $G(d,p)$ for some suitably chosen probability $p=p(d,g)$, one can show the existence of an $\F$-free graph of order $d$ and $\Theta(d^{1+1/(g-3)})$ edges. 
This provides a lower bound for $\ex(d,\F)$.
By Theorem~\ref{thm:main}, we have that every $d$-regular graph 
$G$ has a spanning subgraph with girth at least $g$ and minimum degree 
$\Omega(d^{1/(g-3)})$. This improves the result of 
Kun~\cite{k2013}. 

\item[-] Let $\F=\F_6=\{C_3,C_4,C_5\}$. From the constructions of extremal $C_4$-free graphs 
provided by Erd\H os, R\' enyi and S\' os~\cite{ers1966} and Brown~\cite{b1966} 
one can obtain an $\F$-free graph of order $d$ and $\Theta(d^{3/2})$ edges. 
Again, by Theorem~\ref{thm:main}, for every $d$-regular graph $G$ we 
can show the existence of a spanning subgraph with girth at least $6$ and 
minimum degree $\Omega(\sqrt{d})$. Extremal constructions for graphs with girth 
at least $8$ and $12$ are also known~\cite{lv2005,luw1999}. From them we can 
obtain tight explicit lower bounds for $\F_8=\{C_3,\dots,C_7\}$ and 
$\F_{12}=\{C_3,\dots,C_{11}\}$.

\item[-] Let $\F=\{K_{a,b}\}$ with $a\leq b$. Since $K_{a,b}$ has diameter two, 
each locally injective homomorphism of $K_{a,b}$ onto a graph $\G$, is also 
injective. Thus, if $\hom^*(K_{a,b},\G)=0$, then $\G$ is $K_{a,b}$-free, and 
$\F$ is closed. In particular, any family composed by complete bipartite graphs 
is closed. It is conjectured that $\ex(d,K_{a,b})=\Theta(n^{2-1/a})$. While the 
upper bound has been proved for all values of $a$ and $b$~(Kov{\'a}ri, S{\'o}s 
and Tur{\'a}n~\cite{kst1954}), the lower bound is still wide open. However, 
it is known to be true in the following cases: $a=2$ and $b\geq 2$, $a=3$ and 
$b\geq 3$ (Brown~\cite{b1966}), and $b>(a-1)!$ (Koll\'ar, R\'onyai and 
Szab\'o \cite{krs1996} and Alon, R\'onyai and 
Szab\'o~\cite{ars1999}). Using these results we can get tight explicit lower 
bounds on the minimum degree of the largest subgraph of a $d$-regular graph that induces no 
complete bipartite subgraph of a given size.

\item[-]  Let $\F=\{Q_3\}$, where $Q_s$ is the $s$-dimensional hypercube. The 
family $\F$ is not closed since $Q_3$ admits a locally injective homomorphism to $K_4$, that is not injective. As before, the family $\F'=\{K_4, Q_3\}$ is closed. Since $\chi(K_4)=4$, we have $\ex(d,\F)=\Theta(\ex(d,\F'))$ and Conjecture~\ref{conj:fkp} is true for $\F$. It is conjectured 
in~\cite{es1970} that $\ex(d,\F)=\Theta(d^{8/5})$, but no better lower bound that 
$\ex(d,\F)\geq \ex(d,C_4) = \Omega(d^{3/2})$ is known.
\end{itemize}

Theorem~\ref{thm:main} solves in the affirmative 
Conjecture~\ref{conj:fkp} for families of graphs satisfying a ``local'' 
condition. However, there are many families the contain bipartite graphs and that are not closed. 
In order to solve the conjecture for every family of graph $\F$, one 
needs to extend the idea of local injectivity in Theorem~\ref{thm:more_general} to injectivity. In terms of 
colorings, it would suffice to prove the existence of a spanning subgraph $H$ with relatively
large minimum degree and a coloring $\chi$ such that \emph{all} copies of $F$ in 
$H$ are rainbow. This is stronger than the rainbow condition in 
the neighborhoods of $H$ that we impose here.

Finally, Theorems~\ref{thm:more_general} and~\ref{thm:main} study the case where $G$ is 
$d$-regular. Similar results in terms of the maximum and minimum degree have 
been given in~\cite{fkp2014}. We believe that the same techniques used here 
could be extended to the non regular case, possibly adding a mild condition between the maximum and the minimum degree.
\vspace{0.5cm}

\textbf{Related work.} Conjecture~\ref{conj:fkp} is closely related to  the following very general question: 
given a graph parameter $\rho$, a value $k$ and a family of graphs $\F$, determine the largest value of $\ell$ such 
that for every graph $G$ with $\rho(G)\geq k$ there exists an $\F$-free subgraph $H$ of $G$ with $\rho(H)\geq \ell$. 
Here we list some interesting results for other important graph parameters:
\begin{itemize}
\item[-] Let $\rho$ be the average degree. Thomassen~\cite{t1983} conjectured that every graph with average degree at least $d$ has an $\F_g$-free subgraph with average degree $f(d)$, for some $f(d)\to\infty$ as $d\to \infty$. K\"uhn and Osthus~\cite{ko2004} showed that Thomassen's conjecture is true for $g=6$, but the general conjecture is still wide open. 

\item[-] Let $\rho$ be the number of edges. Bollob\'as and Erd\H os asked this problem for the case $\F=\{C_4\}$ in a workshop in 1966~\cite{e1968}. The problem was rediscovered by Foucaud, Krivelevich and Perarnau~\cite{fkp2014} and they provided lower bounds for the case $\F_g$ that are tight up to logarithmic factors.  Conlon, Fox and Sudakov~\cite{cfs2014,cfs2015} proved tight lower bounds when $\F$ is composed by complete bipartite graphs. In particular, this answers the question of Bollob\'as and Erd\H os.

\item[-] Let $\rho$ be the vertex-connectivity. Thomassen~\cite{t1989} also conjectured that every $k$-connected graph has a bipartite subgraph ($\F$-free for $\F=\{C_3,C_5,C_7,\dots\}$) with connectivity $f(k)$, for some $f(k)\to\infty$ as $k\to \infty$.  A first step towards the proof of this conjecture is the result of Delcourt and Ferber~\cite{df2014}.
\end{itemize}

\section{Proof of Theorem~\ref{thm:main} assuming Theorem~\ref{thm:more_general}}

First, we may assume that the chromatic number of $\F$ is $2$. Otherwise, Theorem~\ref{thm:more_general} can be easily proven (see Proposition~1 in~\cite{fkp2014} for a stronger version of it). We can also assume that $\F$ does not contain any forest. If $\F$ contains a forest $T$, since $\ex(d,\F)\leq \ex(d,T)=O(d)$, then Theorem~\ref{thm:more_general} is trivially true.

Recall that for every $\beta\geq 1$, we say that $\G$ is \emph{$\beta$-almost regular} if $\Delta(\G)\leq \beta\delta(\G)$.
Erd\H os and Simonovits~\cite{es1970} showed that for every $\gamma\in (0,1)$, there exists $\beta\geq 1$ such that the following holds for every family of graphs $\F$ and every positive integer $m$: if every $\F$-free $\beta$-almost regular graph of order $m$ has at most $O(m^{1+\gamma})$ edges, then $\ex(m,\F)=O(m^{1+\gamma})$, where the asymptotic notation here corresponds to $m\to \infty$. 
Otherwise stated, for every $\F$ and every $m$, there exists a $\beta$-almost regular graph $\G$ on $m$ vertices and $\Omega(\ex(m,\F))$ edges.
In particular, $\G$ satisfies $\delta(\G)=\Omega(\ex(m,\F)/m)$.

A classic result of Erd\H os~\cite{e1959,e1961} states that if $\F$ does not contain any forest, then there exists a 
constant $\eps_0=\eps_0(\F)>0$ such that for every large enough $m$, 
$\ex(m,\F)> m^{1+\eps_0}$. 

From these two results, we conclude that for every $\F$ that does not contain a forest there exist $\epsilon>0$ and $\beta\geq 1$ such that for every large enough $m$, there is an $\F$-free $\beta$-almost regular graph $\G$ on $m$ vertices with $\delta(\G)=\Omega(\ex(m,\F)/m)> m^{\eps}$.

Since $\F$ is closed by the hypothesis of Theorem~\ref{thm:main}, by the definition of closed, any $\F$-free graph $\G$ also satisfies $\hom^*(F,\G)=0$, for every $F\in \F$.

Given the family $\F$, the choice of $\a\geq 25\beta^2$ and setting $m=\a d$, we note that
$$
\delta(\G)=\Omega\left(\frac{\ex(\a d,\F)}{\a d}\right) = \Omega\left(\frac{\ex(d,\F)}{d}\right)\;.
$$
Thus, we can use $\G$ in Theorem~\ref{thm:more_general} to find a spanning $\F$-free subgraph $H$ of $G$ with minimum degree  $\delta(H)=\Omega\left(\frac{\ex(d,\F)}{d}\right)$, concluding the proof of Theorem~\ref{thm:main}.

\section{Notation and Probabilistic Tools}

For every $v\in V(G)$  we denote by $N_G(v)$ the set of vertices adjacent to 
$v$ in $G$, by $d_G(v)=|N_G(v)|$ the degree of $v$ in $G$ and by $N^2_G(v)$ the 
set of vertices at distance two from $v$ in $G$. If the graph $G$ is clear from 
the context, we use $N(v)$, $d(v)$ and $N^2(v)$. We denote by $\Delta(G)$ and 
by $\delta(G)$ the maximum and the minimum degree of $G$ respectively. We denote by $\chi$ a 
vertex (partial) coloring of a graph $G$. Throughout the paper, we identify the color palette with the vertices of a graph $\G$ of order $\a d$ and talk indistinctly of an $\a d$ coloring or a $V(\G)$ coloring. 
We use $\chi(G)$ to denote the 
chromatic number of $G$, and for every family of graphs $\F$, we also use 
$\chi(\F)=\min_{F\in\F} \chi(F)$. For every vertex $v\in V$, we denote by 
$\chi(v)$ its color and for every set $S\subseteq V(G)$, by $\chi(S)$ the set 
of colors appearing in $S$. We call $S\subseteq V$  \emph{rainbow} if 
$\chi(u)\neq \chi(v)$ for every $u,v\in S$, $u\neq v$, that were assigned a 
color by $\chi$.

Below, we introduce some standard tools from the probabilistic method that can be 
found in~\cite{mr2002} and that we will be of use in the proof of Theorem~\ref{thm:more_general}.

\begin{lemma}[Chernoff's inequality]\label{lem:chernoff}
	For any $0\leq t\leq np$:
$$
\bP(|\emph{Bin}(n,p)-np|>t)<2 e^{-t^2/3np}\;.
$$
Moreover, for every $t\geq 7np$
$$
\bP(\emph{Bin}(n,p)>t)<e^{-t}\;.
$$
\end{lemma}

\begin{lemma}[Talagrand's inequality]\label{lem:talagrand}
Let $X$ be a nonnegative random variable not identically $0$, which is 
determined by $n$ independent trials $T_1,\dots,T_n$ and satisfying the 
following for some $c_1,c_2>0$:
\begin{itemize}
\item[-] changing the outcome of any one trial can affect $X$ by at most $c_1$, and
\item[-] for any $s$, if $X\geq s$ then there is a set of at most $c_2 s$ trials 
whose outcomes certify that $X\geq s$,
\end{itemize}
then for any $0\leq t\leq \bE(X)$,
$$
\bP(|X-\bE(X)|>t+60c_1\sqrt{c_2\bE(X)})\leq 4 e^{-t^2/(8c_1^2 c_2 \bE(X))}\;.
$$
\end{lemma}

\begin{lemma}[McDiarmid's inequality]\label{lem:mcdiarmid}
Let $X$ be a nonnegative random variable not identically $0$, which is 
determined by $n$ independent trials $T_1,\dots,T_n$ and $m$ independent uniform random
permutations $\pi_1,\dots,\pi_m$, and satisfying the following for some 
$c_1,c_2>0$:
\begin{itemize}
\item[-] changing the outcome of any one trial can affect $X$ by at most $c_1$,
\item[-] interchanging two elements in any one permutation can affect $X$ by at 
most $c_1$, and
\item[-] for any $s$, if $X\geq s$ then there is a set of at most $c_2 s$ trials 
whose outcomes certify that $X\geq s$,
\end{itemize}
then for any $0\leq t\leq \bE(X)$,
$$
\bP(|X-\bE(X)|>t+60c_1\sqrt{c_2\bE(X)})\leq 4 e^{-t^2/(8c_1^2 c_2 \bE(X))}\;.
$$
\end{lemma}

\begin{lemma}[Lov\'asz Local Lemma]\label{lem:LLL}
Consider a set $\mathcal{E}$ of events such that for each $E\in \mathcal{E}$
\begin{itemize}
\item[-] $\bP(E)\leq p< 1$, and
\item[-] $E$ is mutually independent from the set of all but at most $D$ of other 
events.
\end{itemize}
If $4pD\leq 1$, then with positive probability, none of the events in 
$\mathcal{E}$ occur.
\end{lemma}

\section{Proof of Theorem~\ref{thm:more_general}}

The rest of the paper is devoted to the proof of Theorem~\ref{thm:more_general}.
Throughout this section we will fix a family of graphs $\F$, $\epsilon>0$, $\beta\geq 1$ and $\a \geq 25\beta^2$. In order for some inequalities to hold, we will require $d$ to be large enough with respect to the previous parameters. Recall that, by assumption, there exists a graph $\G$ on $\a d$ vertices that satisfies the conditions of Theorem~\ref{thm:more_general}. We first describe the main ideas of the proof.
\subsection{Outline of the proof}
First of all, we choose a large bipartite subgraph $G'$ of $G$, by selecting a maximum edge-cut and keeping the edges in it. Observe that every vertex in $G'$ has degree at least $d/2$. Let $A$ and $B$ be the two stable sets of $G'$.

The goal of the proof is to select a spanning subgraph $H$ of $G'$ and a $V(\G)$ coloring $\chi$ of $H$ that satisfies the following properties
\begin{itemize}
\item[-]  for every $v\in V(H)$, $d_H(v)=\Omega(\delta(\G))$,

\item[-]  for every $v\in V(H)$, $N_H(v)$ is rainbow, and

\item[-]  for every $uv\in E(H)$, $\chi(u)\chi(v)\in E(\G)$.
\end{itemize}
From these properties, it will be easy to deduce that the subgraph $H$ satisfies the 
statement of Theorem~\ref{thm:more_general}. This is done at the end of the paper.

The subgraph $H$ is selected by deleting some edges of $G'$ depending on the color assigned to its vertices. Here and throughout the proof of the theorem, we always identify the set of $\a d$ colors with $V(\G)$.
\vspace{0.1cm}

\noindent We design a coloring procedure that is divided in two phases.

In Phase $I$ we color the vertices of $A$ and delete some of the edges of $G'$. This is done in two steps: the first one being random and the second one, deterministic. We obtain a spanning subgraph $H'$ of $G$ and a coloring $\chi'$ of $A$ with some desirable properties~(see Lemma~\ref{lem:after_phase_I} for a precise description).

In Phase $II$ we color the vertices of $B$ and delete some of the edges of $H'$. In contrast to Phase $I$, here we perform a randomized iterative coloring procedure to color most of the vertices of $B$. We also complete the coloring of $B$ with a deterministic step. After Phase $II$, we obtain the spanning subgraph $H$ of $G$ and the coloring $\chi$ of $H$ with the properties described above.

\subsection{Phase $I$: Coloring $A$}
For every graph $G$, partial coloring $\chi$ of $G$ and $v\in V$, we define
$$
\Bad(v,\chi,G)= |\{u\in N_G(v):\; \exists v'\in N_G(u)\setminus\{v\}\mbox{ and } 
\chi(v')=\chi(v)  \}|\;,
$$
that is, the number of vertices $u\in N_G(v)$ such that $\chi(v)$ appears more 
than once in $N_G(u)$. This quantity will be crucial throughout the paper.

We will color the set $A$ as follows:
\begin{enumerate}
\item For every $a\in A$, let $\chi'_0(a)=c$, where $c\in V(\G)$ is chosen 
independently and uniformly at random. 
\item Uncolor $a\in A$ if
$$\Bad(a,\chi'_0,G')\geq \frac{d}{\sqrt{\a}}\;.$$
 Let $A_1$ be the set of uncolored vertices in $A$.

\item Delete all the edges between $a\in A\setminus A_1$ and $N_{G'}(a)$ that may cause 
conflicts, i.e. delete $ab$ if $\exists a'\in N_{G'}(b)\setminus\{a\}$ such that $\chi'_0(a')=\chi'_0(a)$. Let $H'_0$ be  the subgraph obtained 
by removing these edges from $G'$.

\item Consider an arbitrary order of the vertices in $
A_1=(a_1,\dots,a_{s})$, where $s=|A_1|$.

\item For every $i$ from $1$ to $s$,
\begin{enumerate}

\item Assign to $a_i$ the color  in $V(\G)$ that minimizes 
$\Bad(a_i,\chi'_{i-1},H'_{i-1} )$.
Let $\chi'_i$ be the partial coloring of $A$ obtained from $\chi'_{i-1}$ and 
the colored vertex $a_i$.

\item Delete all the edges  between $a_i$ and $N_{H'_{i-1}}(a_i)$ that may 
cause conflicts, i.e. delete $a_i b$ if $\exists a'\in N_{H'_{i-1}}(b)\setminus\{a_i\}$ such that  and $\chi'_{i}(a')=\chi'_i(a_i)$. Let $H'_{i}$ be 
the subgraph obtained by removing these edges from $H'_{i-1}$.
\end{enumerate}
\item Let $\chi'=\chi'_s$ and $H'=H'_s$.

\end{enumerate}

For a random map $\chi:\,V(G) \to V(\G)$, let $p$ be the probability that 
an edge $uv\in E(G)$ gets mapped to an edge $\chi(u)\chi(v)\in E(\G)$. Then we have,
$$
p^-:=\frac{\delta(\G)}{\a d}\leq p=\frac{2|E(\G)|}{(\a d)^2}\leq \frac{\Delta(\G)}{\a d}=:p^+  \;.
$$
Observe that since $\G$ is $\beta$-almost regular, $p^+ \leq \beta p^-$.

The next lemma is going to be useful to prove that $\chi'$ has some good properties with respect to each  $b\in B$.
\begin{lemma}\label{lem:many_goodneighs}
Let $\alpha\geq 25$ and let $d$ be large enough. Let $\G$ be a graph on $\alpha d$ vertices.
Let $G'$ be a bipartite graph with stable sets $A$ and $B$, maximum degree at most $d$ and minimum degree at least $d/2$. Let 
$\chi$ be a random coloring of $A$ where each vertex is assigned a color from 
$V(\G)$ independently and uniformly at random. For every $b\in B$ and $c\in 
V(\G)$, let $W_{b,c}$ be the number of vertices $a\in N_{G'}(b)$ such that
\begin{enumerate}
\item[1.]  $a$ is the only vertex with color $\chi(a)$ in $N_{G'}(b)$,
\item[2.] $\Bad(a,\chi,G')\leq \frac{d}{\sqrt{\a}}$, and
\item[3.]  $\chi(a)c\in E(\G)$.
\end{enumerate}
Then,
$$ 
\bP\left( W_{b,c}\leq \frac{\delta(\G)}{4\a}\right)= e^{-\Omega(\delta(\G))}\;.
$$
\end{lemma}
\begin{proof}
Let $X$ be the number of neighbors $a\in N(b)$ such that at least one of these 
two conditions holds
\begin{itemize}
\item[$(F1)$] \hspace{0.3cm} $\exists a'\in N(b)\setminus\{a\}$ with
$\chi(a')=\chi(a)$.
\item[$(F2)$]  \hspace{0.3cm} 
$
|\{b'\in N(a)\setminus\{b\}:\; \exists a'\in N(b')\setminus N(b)\mbox{ and } \chi(a')=\chi(a)  \}|\geq \frac{d}{\sqrt{\a}}-1\;.
$
\end{itemize}
Observe that if $a$ does not satisfy conditions $(F1)$ and $(F2)$, then there 
are at most $ \frac{d}{\sqrt{\a}}$ vertices $b'\in N(a)$ (this might include 
$b'=b$) such that there exists $a'\in N(b')\setminus\{a\}$ with $\chi(a')=\chi(a)$. In other 
words, $\Bad(a,\chi,G')\leq \frac{d}{\sqrt{\a}}$~(condition $2$ of the lemma). Moreover, condition $(F1)$ 
also implies that $a$ is the only vertex with color $\chi(a)$ in $N(b)$~(condition $1$ of the lemma).
 
We first show that, with high probability, $X$ is not too large. We will prove a stronger statement: $X$ is concentrated around its expected value. Indeed, we will fix a coloring of 
$A\setminus N(b)$ and prove that, conditional on any such coloring, $X$ is concentrated.

For the event $(F1)$, since $d(a)\leq d$ and there are $|V(\G)|=\a d$ colors, the probability that a vertex $a$ receives the same color as one of the other neighbors of $b$ is at most $1/\a$.

We call a color $c'$ \emph{dangerous} for $a\in N(b)$ if $(F2)$ is satisfied 
when $\chi(a)=c'$. The fact that $c'$ is dangerous for $a$ depends only on 
the coloring of $A\setminus N(b)$ which is already fixed. In particular, since 
there are at most $d^2$ edges $a'b'$ with $b'\in N(a)$, the number of dangerous 
color classes is at most $\frac{d^2}{(d/\sqrt{\alpha})-1}\leq 2\sqrt{\a} d$. 
Thus, the probability that a vertex $a\in N(b)$ satisfies the event $(F2)$ is 
at most $2/\sqrt{\a}$.

Using a union bound for the events $(F1)$ and $(F2)$ we obtain
$$
\bE(X \mid \chi(A\setminus N(b)))\leq  \left(\frac{1}{\a}+ \frac{2}{\sqrt{\a}}\right) d(b) \leq 
\frac{3d(b)}{\sqrt{\a}}\;.
$$

Since we have conditioned on the coloring given to $A\setminus N(b)$, $X$ only 
depends on the colors assigned to $N(b)$.
Changing the color of a vertex $a\in N(b)$ from $c'$ to $c''$ can change the value of $X$ by at most 
$2$. This change can create or destroy at most two vertices 
with the property $(F1)$. It can also be the case that $c''$ is a dangerous 
color for $a$ and that $c'$ is not (or vice versa); in this case the 
change is by at most $1$.

Furthermore, if we have $X\geq s$ and given that the coloring in $A\setminus 
N(b)$ has been fixed previously, there exists a set of $s$ choices of colors 
that certify $X\geq s$; if a vertex $a\in N(b)$ satisfies $(F1)$ it suffices to 
reveal all the vertices in $N(b)$ with color $\chi(a)$ (and all them count for $X$), and if it satisfies 
$(F2)$, it suffices to reveal $\chi(a)$, since the fact that it is a dangerous color is certified by the coloring of $A\setminus N(b)$, which is fixed.

Since $\alpha\geq 25$, by Talagrand's inequality with $c_1=2$ and $c_2=1$, we have
\begin{align}\label{eq:1}
\bP\left( X\geq \frac{4d(b)}{\sqrt{\a}}\mid \chi(A\setminus N(b))\right)&\leq 
\bP\left( |X-\bE(X)|\geq \frac{d(b)}{\sqrt{\a}}\mid \chi(A\setminus N(b))\right) 
= e^{-\Omega(d)}\;.
\end{align}
Since the previous inequality holds conditional on any coloring of $A\setminus 
N(b)$, it also holds unconditionally.

We now count the number of vertices $a\in N(b)$ for which neither (F1) nor (F2) hold and that satisfy $\chi(a)c\in E(\G)$. Such vertices are in $W_{b,c}$ since they fulfill conditions $1$, $2$ and $3$.
Observe that coloring each vertex in $A$ with a color chosen independently and 
uniformly at random in $V(\G)$ is equivalent to color $A$ in the same way and 
then permute the color classes according to a permutation $\pi$ of length 
$\a d=|V(\G)|$ chosen uniformly at random. These two steps can be done independently 
and thus we can analyse them separately. In the first step the color 
classes are set, while, in the second one each color class is assigned a particular color 
according to $\pi$. Observe that the condition $\chi(a)c\in E(\G)$ only depends 
on the second step, while $(F1)$ and $(F2)$ only 
depend on the partition induced by the color classes.

Let $M$ be the set of vertices $a\in N(b)$ such that, after the first step, conditions $(F1)$ and 
$(F2)$ are not satisfied. Notice that every vertex in $M$ receives a different 
color. Since $\pi$ is a uniformly chosen permutation, the set of $|M|$ colors 
assigned to $M$ is chosen uniformly from all the sets of size $|M|\subseteq [\a 
d]$.
Then, for every $I\subseteq N(b)$ with $|I|=i$,
$$
\bE(W_{b,c} | M=I) =\sum_{a\in I} \bP(\chi(a)c\in E(\G)| M=I)= \sum_{a\in I} 
\bP(\chi(a)c\in E(\G))\geq  p^- i\;.
$$
Recall that, by the independence of the two step coloring, to show that 
$W_{b,c}$ is large with high probability, we only need to focus on the second 
step. Once the color classes have been set, a swap between two positions of 
$\pi$ can change $W_{b,c}$ by at most $2$ and if $W_{b,c}\geq s$ we can certify 
it by only revealing the permutation $\pi$. By McDiarmid's inequality, for all 
$0<\gamma<1$
\begin{align}\label{eq:2}
\bP\left( W_{b,c}\leq (1-\gamma)p^- i| M=I\right)&\leq \bP\left( 
|W_{b,c}-\bE(W_{b,c})|\geq \gamma p^- i| M=I \right)\nonumber\\
 &= e^{-\Omega(p^-i)}\;.
\end{align}

Since $d(b)\geq d/2$, we have $\frac{p^-d(b)}{2}\geq \frac{\delta(\G)}{4\a}$. Now, using~(\ref{eq:1}) and~(\ref{eq:2}), we obtain the desired result
\begin{align*}
\bP\left( W_{b,c}\leq \frac{\delta(\G)}{4\a}\right) &\leq
\bP\left( W_{b,c}\leq \frac{p^- d(b)}{2}\right) \\
&= \sum_{I\subseteq  N(b)} 
\bP\left( W_{b,c}\leq  \frac{p^- d(b)}{2}| M=I\right) \bP(M=I)\\
&\leq \sum_{\substack{I\subseteq N(b)\\ i\geq (1-\frac{4}{\sqrt{\a}})d(b)}} 
\bP\left( W_{b,c}\leq \frac{p^-d(b)}{2}| M=I\right) \bP(M=I)+ e^{-\Omega(d)}\\
&\leq \sum_{\substack{I\subseteq  N(b)\\ i\geq (1-\frac{4}{\sqrt{\a}})d(b)}} 
\bP\left( W_{b,c}\leq \frac{p^-i}{2(1-\frac{4}{\sqrt{\a}})}| M=I\right) 
\bP(M=I)+ e^{-\Omega(d)}\\
&\leq  e^{-\Omega(\delta(\G))} \sum_{\substack{I\subseteq  N(b)\\ i\geq 
(1-\frac{4}{\sqrt{\a}})d(b)}}\bP(M=I)+ e^{-\Omega(d)}\\
&= e^{-\Omega(\delta(\G))}\;.
\end{align*}
\end{proof}

\begin{lemma}\label{lem:after_phase_I}
Let $\epsilon>0$ and $\alpha\geq 25$ and let $d$ be large enough. Let $\G$ be a graph on $\alpha d$ vertices such that $\delta(\G)>d^{\epsilon}$.

With positive probability, Phase $I$ provides a spanning subgraph $H'$ of $G$ and a coloring $\chi'$ of $A$ with the following properties,
\begin{itemize}
\item[$(P1)$] \hspace{0.3cm} for every $a\in A$, $d_{H'}(a)\geq d/4$,

\item[$(P2)$] \hspace{0.3cm} for every $b\in B$, $N_{H'}(b)$ is rainbow, and

\item[$(P3)$] \hspace{0.3cm} for every $b\in B$ and every $c\in V(\G)$, there 
are at least $\delta(\G)/4\a$ vertices $a\in N_{H'}(b)$ such that 
$\chi'(a)c\in E(\G)$.
\end{itemize}
\end{lemma}
\begin{proof}
Let us first show that $(P1)$ is satisfied. If $a\in A\setminus A_1$, we have 
$d_{H'}(a)= d_{H'_0}(a)\geq (1-1/\sqrt{\a})d_{G'}(a)$, since by the choice of 
$\chi'_0(a)$ we deleted at most $d/\sqrt{\a}$ edges incident to it. For every $a\in  A_1$, there is an $1\leq i\leq s$ such that $a=a_i$. Since there are at most $d^2$ edges 
incident to $N_{H_{i-1}'}(a_i)$, there exists a color $c\in V(\G)$ (recall that 
$|V(\G)|=\a d$) such that if $\chi'_i(a_i)=c$
$$
\Bad(a_i,\chi'_{i},H'_{i-1} )\!=\! |\{b\in N_{H'_{i-1}}(a_i):\; \exists a'\in 
N_{H'_{i-1}}(b)\setminus\{a_i\}\mbox{ and } \chi'_{i}(a')=\chi'_{i}(a_i)  \}|\leq 
\frac{d}{\a}\;.
$$
Since in Phase $I$ we only delete edges incident to $a_i$ when its color is 
assigned, $d_{H'}(a_i)=d_{H'_{i}}(a_i)\geq (1-1/\a)d_{G'}(a)$ for all $a_i\in 
A_1$. Since $\a\geq 8$ and since $d_{G'}(a)\geq d/2$, for all $a\in A$ we have $d_{H'}(a)\geq d/4$.

Property $(P2)$ also holds deterministically. Consider a vertex $b\in B$. Note that an edge $ab$ can only be deleted in the iteration when $a$ retains its color. Suppose that $a$ retains its color at the $i$-th iteration and that the edge $ab$ is not deleted. Then, the color $\chi'_i(a)$ only appears once in the partial coloring $\chi'_i$ restricted to $N_{H'_{i-1}}(b)$. Moreover, if a vertex $a'\in N_{H'_i}(b)\setminus\{a\}$ retains the color $\chi'_i(a)$ in a further iteration, then the edge $a'b$ is deleted. Thus, $a$ is the only neighbour of $b$ in $H'$ with color $\chi'_i(a)$. Since the choice of $a$ is arbitrary, $N_{H'}(b)$ is rainbow.

Let us show that $(P3)$ holds with positive probability. In order to show it, it suffices to look at the first partial coloring $\chi'_0$.
For every $b\in B$ and $c\in V(\G)$, let $E_{b,c}$ be the event that there are 
at most $\delta(\G)/4\a$ vertices $a\in N_{H'_0}(b)$ such that 
$a$ is the only vertex with color $\chi'_0(a)$ in $N_{H'_0}(b)$, $a$ retains its color (i.e. $\Bad(a,\chi'_0,G')\geq \frac{d}{\sqrt{\a}}$) and 
$\chi'_0(a)c\in E(\G)$.

By applying Lemma~\ref{lem:many_goodneighs} to $G'$ with $\chi=\chi'_0$ and noting that by assumption $\delta(\G)> d^\eps$,  we obtain $\bP( 
E_{b,c} )=  e^{-\Omega(d^\eps)}$. The event $E_{b,c}$ 
is mutually independent from the other events $E_{b',c'}$ where both $b'$ and 
$c'$ are at distance larger than $6$ from $b$ and $c$ in the respective graphs 
$G$ and $\G$. Then, each event is mutually independent from all but at most 
$d^6\Delta(\G)^6\leq d^{12}$ other events. Since 
$e^{-\Omega(d^\eps)}=O(d^{-13})$, provided that $d$ is large enough with respect to $\epsilon$ and $\alpha$, we can use the Local Lemma to show that none of the events in 
$E_{b,c}$ holds with positive probability.

By coloring the vertices $a_i\in A_1$ and deleting edges incident to 
$a_i$, one can neither decrease the degree of $b$ in $A\setminus A_1$ nor change the fact 
that a given $a\in A\setminus A_1$ is the only neighbor of $b$ with color $\chi'_0(a)$. 
Thus, with positive probability $(P3)$ is satisfied.
\end{proof}

\subsection{Phase $II$: Coloring $B$}
In the second coloring phase it will be convenient to redefine $\Bad(v,\chi,G)$ in a way that 
it takes into account the compatibilities between colors given by the edges of 
$\G$. For every graph $G$, every set $F\subseteq E(G)$ (we say that the edges in $F$ are \emph{activated}), every partial coloring $\chi$ of $G$, every $v\in V$ and every graph $\G$, we define
\begin{align*}
\Bad_\G(v,\chi,G)= \left|\left\{u\in N_G(v):
\begin{array}{c}
    \chi(u)\chi(v)\in E(\G),\exists v'\in N_G(u)\setminus\{v\},\\ \hspace{0.3cm}
   uv'\in F\text{ and } \chi(v')=\chi(v) \\
  \end{array}
\right\}\right| \;,
\end{align*}

We will color $B$ as follows:
\begin{enumerate}

\item Let $B_0=B$, $H_0=H'$, $\chi_0=\chi'$ and $i=1$. For every $a\in A$, let 
$B_0(a)=B_0\cap N_{H_0}(a)$.

\item While there exists $a\in A$ with $|B_{i-1}(a)|> d^{\epsilon/2}$,
\begin{enumerate}

\item For every edge in $H_{i-1}$, activate it independently with probability $1/\a$.

\item Construct $\chi_i$ as follows. For every $v\in A\cup (B\setminus 
B_{i-1})$, let $\chi_i(v)=\chi_{i-1}(v)$ and for every $b\in B_{i-1}$, let 
$\chi_i(b)=c$, where $c$ is chosen independently and uniformly at 
random in $V(\G)$.

\item Uncolor $b\in B_{i-1}$ if
$$
\Bad_\G(b,\chi_i,H_{i-1})\geq \frac{\delta(\G)}{8\alpha}\;.
$$
Let $B_{i}$ be the set of uncolored vertices.

\item Construct $H_i$ from $H_{i-1}$  by deleting the following edges: for 
every $b\in B_{i-1}\setminus B_i$, delete all the edges 
$ab$ in $H_{i-1}$ such that either can cause conflicts (there exists a vertex $ 
b'\in N_{H_{i-1}}(a)\setminus \{b\}$ such that $ab'$ is activated, $\chi_i(a)\chi_i(b)\in E(\G)$  and $\chi_i(b')=\chi_i(b)$), 
$\chi_i(a)\chi_i(b)\not\in E(\G)$ or $ab$ has not been activated.

\end{enumerate}
Let $B_i(a)=B_i\cap N_{H_{i}}(a)$ and increase $i$ by one.

\item Consider an arbitrary order of the vertices in $B_{\tau}=(b_{\tau+1},\dots,b_{t})$, where $t=|B_\tau|+\tau$.

\item For every $i$ from $\tau+1$ to $t$,
\begin{enumerate}
\item Activate all the edges of $H_{i-1}$.
\item Assign to $b_i$ the color in $\in V(\G)$ that minimizes  
$\Bad_\G(b_i,\chi_{i-1},H_{i-1})$.
Let $\chi_i$ be the partial coloring of $V(G)$ obtained from $\chi_{i-1}$ and 
the colored vertex $b_i$.

\item Construct $H_i$ from $H_{i-1}$  by deleting all the edges $ab_i$
such that either can cause conflicts (there exists a vertex $
b'\in N_{H_{i-1}}(a)\setminus \{b_i\}$ such that $\chi_i(a)\chi_i(b_i)\in E(\G)$ and $\chi_i(b')=\chi_i(b_i)$) or $\chi_i(a)\chi_i(b_i)\not\in E(\G)$.

\end{enumerate}

\item Let $\chi=\chi_t$ and $H=H_t$.
\end{enumerate}

We define $\tau$ to be the smallest $i$ such that $|B_i(a)|\leq 
d^{\eps/2}$ for every $a\in A$; that is, the number of iterations in step $2$ of Phase $II$. Observe that $\tau$ is a random variable.

\begin{lemma}\label{lem:after_phase_II}

With positive probability, Phase $II$ of the coloring procedure ends and 
provides a spanning subgraph $H$ of $G$ and a coloring $\chi$ of $H$ with the 
following properties,
\begin{itemize}
\item[$(Q1)$]  \hspace{0.3cm} for every $v\in V(H)$, $d_H(v)\geq 
\frac{\delta(\G)}{16\a^2}$,

\item [$(Q2)$]  \hspace{0.3cm} for every $v\in V(H)$, $N_H(v)$ is rainbow, and

\item [$(Q3)$]  \hspace{0.3cm} for every $uv\in E(H)$, $\chi(u)\chi(v)\in 
E(\G)$.
\end{itemize}
\end{lemma}
Let $\GN_i(a)$ be number of vertices in $b\in B_{i-1}(a)\setminus B_i(a)$ such that $ab$ is active in $H_{i-1}$ and $\chi_i(a)\chi_i(b)\in E(\G)$.
Recall that $p^+= \Delta(\G)/\a d$.

In order to prove Lemma~\ref{lem:after_phase_II} we will make sure that, for every $1\leq i<\tau$, the following three 
conditions are satisfied:
\begin{itemize}
\item[$(C1)$] \hspace{0.3cm}  for every $a\in A$, $\GN_{i}(a)\leq \max 
\left\{\frac{2}{\a}\cdot p^+|B_{i-1}(a)|, d^{\eps/2}\right\}$,
\item[$(C2)$] \hspace{0.3cm}  for every $a\in A$, $|B_i(a)|\leq 
\frac{d}{\a^{i/2}}$, and
\item[$(C3)$] \hspace{0.3cm}  for every $b\in B_{i-1}\setminus B_i$, 
$d_{H_i}(b)\geq \frac{\delta(\G)}{16\a^2}$.
\end{itemize}
In words, $(C1)$ ensures that the degree of $a$ in $B\setminus B_\tau$ is not 
too large; $(C2)$ implies that the number of uncolored vertices in each 
neighborhood of $A$ decreases exponentially with the number of iterations; 
and $(C3)$ ensures that $b\in B\setminus B_\tau$ has a large minimum degree in $H$. We will use $(C1)$ and $(C2)$ to prove that $b\in B_\tau$ also has a large minimum degree in $H$.

In particular, we will require that the following condition is satisfied after the first 
iteration of Phase $II$:
\begin{itemize}
\item[$(C4)$]  \hspace{0.3cm}  for every $a\in A$, the degree of  $a$ to 
$B\setminus B_1$ in $H_1$  is at least $\frac{\delta(\G)}{16\a^2}$.
\end{itemize}
Condition $(C4)$ ensures that the degree of $A$ in $H$ is large.
\vspace{0.2cm}

Let us show that the probability that any of these four conditions is 
violated, is exponentially small in terms of $d$. 
In order to control the 
condition $(C1)$, for every $a\in A$, we define $D^i_1(a)$ to be the event that 
$\GN_{i}(a)\geq  \max\left\{\frac{2}{\a}\cdot p^+|B_{i-1}(a)|, 
d^{\eps/2}\right\}$.
\begin{lemma}\label{lem:1}
For every $1\leq i<\tau$ and for every $a\in A$, we have
$$
\bP(D^i_1(a))  = e^{-\Omega(d^{\eps/4})} \;.
$$
\end{lemma}

\begin{proof}
Each edge $ab$ is activated independently with probability $1/\a$. Moreover, by definition of $p^+$, and independently from the fact that $ab$ has been activated, when $b$ is assigned a random color $c$, it satisfies $\chi_i(a)c\in E(\G)$ with probability at most $p^+$.  Since the choice of color 
for each $b\in B_{i-1}(a)$ is done independently, $\GN_i(a)$ follows a binomial 
distribution with $|B_{i-1}(a)|$ trials and probability at most 
$p^+/\a$. In particular, $\bE(\GN_i(a))\leq \frac{1}{\a}\cdot p^+|B_{i-1}(a)|$.

Suppose first that $|B_{i-1}(a)|\geq \frac{\a d^{\eps/4}}{p^+}$.  By Chernoff's 
inequality,
\begin{align*}
\bP\left(\GN_i(a) \geq \frac{2}{\a}\cdot p^+|B_{i-1}(a)| \right) &\leq 
\bP\left(|\GN_i(a)-\bE(\GN_i(a))|\geq \frac{1}{\a}\cdot p^+|B_{i-1}(a)| \right) \\
&= e^{-\Omega(p^+|B_{i-1}(a)|)} = e^{-\Omega(d^{\eps/4})}\;.
\end{align*}

Suppose now that $|B_{i-1}(a)|<\frac{\a d^{\eps/4}}{p^+}$. Then, $\bE(\GN_i(a))< 
d^{\eps/4}$. Provided that $d$ is large enough, $d^{\eps/2}\geq 7\bE(\GN_i(a))$ and by Chernoff's inequality,
\begin{align*}
\bP\left(\GN_i(a) \geq d^{\eps/2} \right) &=e^{-\Omega(d^{\eps/2})}\;.
\end{align*}
\end{proof}

In order to control the condition $(C2)$, for every $a\in A$, we define 
$D^i_2(a)$ to be the event that  $|B_i(a)|\geq \frac{d}{\a^{i/2}}$. 
\begin{lemma}\label{lem:2}
For every $1\leq i<\tau$ and for every $a\in A$, if $(C1)$ and $(C2)$ hold for 
every $1\leq j<i$, then
$$
\bP(D^i_2(a))  = e^{-\Omega(d^{\eps/4})} \;.
$$
\end{lemma}
\begin{proof}
We call a color $c$ \emph{dangerous} for $b\in B_{i-1}(a)$ if the number of 
vertices $a'\in N_{H_{i-1}}(b)\setminus\{a\}$ such that $\chi_i(a')\chi_i(b)\in 
E(\G)$ and there exists $b'\in N_{H_{i-1}}(a')\setminus B_{i-1}(a)$ with $a'b'$ activated and $\chi_i(b')=c$, is at least $\frac{\delta(\G)}{8\a}-1$. The fact that $c$ is a dangerous color is fully determined by the coloring of $B\setminus B_{i-1}(a)$ and the choice of activated edges between $A$ and $B\setminus B_{i-1}(a)$.

We now show that conditional on a certain event that holds with very high 
probability, for every $b\in B_{i-1}(a)$ there are few dangerous colors. 

Let $\Free_i(a')$ be the number of colors $c\in V(\G)$ with
$\chi_i(a')c\in E(\G)$ such that at least one of the following two conditions is satisfied:
\begin{itemize}
\item[-] $c$ does not appear in $N_{H_{i-1}}(a')\setminus B_{i-1}(a)$
\item[-] for every $b'\in N_{H_{i-1}}(a')\setminus B_{i-1}(a)$ with $\chi_i(b')=c$,  $a'b'$ is not activated.
\end{itemize}
In other words, the number of colors such that if we assign one of them to $b$, then $a'$ will not count for $\Bad_J(b,\chi_i,H_{i-1})$.

Let $E$ be the event that $\Free_i(a')\geq \left(1-\frac{1}{\alpha}\right) 
\delta(\G)$, for every $a'\in N^2_{H_{i-1}}(a)$. Observe that the event $E$ 
only depends on the coloring of $A$, the one of $B\setminus B_{i-1}(a)$ and the activated edges between $A$ and $B\setminus B_{i-1}(a)$.  

Next claim shows that the probability of $E$ is very large.
\begin{claim}\label{cla:1}
If $(C1)$ and $(C2)$ are satisfied for every $1\leq j\!< i$, then  $\bP(E) 
\!=\!1-e^{-\Omega(d^{\epsilon/4})}$.
\end{claim}
\begin{proof}
Observe that $\Free_i(a')$ can be controlled with $\GN_j(a')$, for $j\leq i$.
Since $(C1)$ is satisfied for 
every $j<i$, we have that $\GN_j(a')<\max\left\{\frac{2}{\a}\cdot p^+ 
|B_j(a')|,d^{\epsilon/2}\right\}$. Moreover, by Lemma~\ref{cla:1}, $D^i_1(a')$ 
does not hold with probability at least $1-e^{-\Omega(d^{\epsilon/4})}$. Recall that $D^i_1(a')$ is defined as $\GN_i(a')<\max\left\{\frac{2}{\a}\cdot p^+ 
|B_i(a')|,d^{\epsilon/2}\right\}$. A union bound shows that this is true for all 
$a'\in N^2_{H_{i-1}}(a)$ with probability $1-e^{-\Omega(d^{\epsilon/4})}$.

Since $(C2)$ holds for every $j<i$ and $|B_i(a')|\!\leq |B_{i-1}(a')|$, we have that  $\sum_{j=1}^i |B_j(a')|\leq 2d$. Thus, for every $a'\in N^2_{H_{i-1}}(a)$ we obtain
$$
\Free_i(a')\geq \delta(\G) -\sum_{j=1}^i\GN_j(a')\geq \delta(\G) - 
\frac{4}{\a}\cdot p^+ d\geq \left(1-\frac{1}{\alpha}\right) \delta(\G)\;,
$$
where we used that $\alpha\geq 25\beta^2$ in the last inequality. 
We conclude that $E$ is satisfied with probability 
$1-e^{-\Omega(d^{\epsilon/4})}$.
\end{proof}

Moreover, the event $E$ allow us to control the number of dangerous colors for 
the neighbors of $a$.
\begin{claim}\label{cla:2}
If $E$ is satisfied, then  for each $b\in N_{H_{i-1}}(a)$, the number of 
dangerous colors for $b$ is at most $9d$.
\end{claim}
\begin{proof}
Since $E$ implies $\Free_i(a')\geq\left(1-\frac{1}{\a}\right) \delta(\G)$ for 
every $a'\in N_{H_{i-1}}(b)\setminus \{a\}$, we have that there are at most 
$\frac{\delta(\G)}{\a}\cdot  d_{H_{i-1}}(b)$ edges $a'b'$ with 
$\chi_i(a')\chi_i(b')\in E(\G)$.
Recall that if $c$ is dangerous for $b$, then there are at least 
$\frac{\delta(\G)}{8\a}-1$ vertices $a'\in N_{H_{i-1}}(b)\setminus\{a\}$  that are incident 
to an (activated) edge $a'b'$ with $\chi_i(a')\chi_i(b')\in E(\G)$ and $\chi_i(b')=c$.  

Thus, the number of dangerous colors is at most,
$$
\frac{\frac{\delta(\G)}{\a}\cdot d_{H_{i-1}}(b)}{\frac{\delta(\G)}{8\a}-1}\leq 
9d\;.
$$
\end{proof}

Let us prove that the probability $|B_i(a)|$ is large is exponentially small. Let $X$ be the number of vertices $b\in B_{i-1}(a)$ that satisfy 
at least one of these two conditions,
\begin{itemize}
\item[$(F1)$] \hspace{0.3cm} there exists $b'\in N_{H_{i-1}}(a)\setminus\{ 
b\}$ with $\chi_i(b')= \chi_i(b)$,
\item[$(F2)$]  \hspace{0.3cm} $\chi_i(b)$ is a dangerous color for $b$.
\end{itemize}
We claim that $ |B_i(a)|\leq X$. Suppose that $b$ does not satisfy  
$(F1)$ and $(F2)$.  Since $(F1)$ does not hold, all the other vertices in 
$B_{i-1}(a)$ have a color different from $\chi_i(b)$. Together with $(F2)$ not 
holding, we obtain that $\Bad_\G(b,\chi_i,H_{i-1})\leq \frac{\delta(\G)}{8\a}$ 
and $b$ retains its color.

Thus, we will show that $X$ is small with very high probability. Indeed, it 
will suffice to show it in the case that the event $E$ holds. We fix a coloring 
on $B\setminus B_{i-1}(a)$ and an activation of edges between $A$ and $B\setminus B_{i-1}(a)$ that is compatible with $E$ and show that, 
conditional on that, $X$ is concentrated. Recall that the choice of the coloring 
of $B\setminus B_{i-1}(a)$ and of the activation of the edges between $A$ and $B\setminus B_{i-1}(a)$ fully determines whether $E$ is satisfied.

For $b\in B_{i-1}(a)$, the probability that $\chi_i(b)$ is assigned a color that already appears in $N_{H_{i-1}}(a)$ (that is, of condition $(F1)$) is at most $1/\a$, since $d(a)\leq d$ and there are $\a d$ colors.
Since $E$ is satisfied, by Claim~\ref{cla:2}, for each $b\in B_{i-1}(a)$ there are at most $9d$ dangerous colors. Thus,  the probability that $b$ is assigned a dangerous color is at most $9/\a$.

By the hypothesis of the lemma, $(C2)$ holds for $i-1$ and there are  
$|B_{i-1}(a)|\leq d/\a^{(i-1)/2}$ candidates for $X$. Hence, conditional on $\chi_i(B\setminus B_{i-1}(a))$ and on the activation of edges,
$$
\bE(X)\leq \frac{10}{\a}\cdot |B_{i-1}(a)|\leq 
10\a^{-1/2}\cdot\frac{d}{\a^{i/2}}\;.
$$

As in the proof of Lemma~\ref{lem:many_goodneighs}, by changing the color of a 
vertex $b\in B_{i-1}(a)$ one can change $X$ by at most $2$ and, if $X\geq s$, then
there exists a set of $s$ colored vertices that certifies $X\geq s$. By applying 
Talagrand's inequality to $X$ with $c_1=2$ and $c_2=1$, and conditional on  $\chi_i(B\setminus B_{i-1}(a))$ and on the activation of edges,
\begin{align*}
\bP\left(D^i_2(a)\right)
&\leq\bP\left(X\geq \frac{d}{\a^{i/2}}\right) 
\leq \bP\left(|X-\bE(X)|> \left(1-10\a^{-1/2}\right) \frac{d}{\a^{i/2}} \right)\\
&= e^{-\Omega(d/\a^{i/2})}=e^{-\Omega(d^{\eps/2})}\;,
\end{align*}
if $\a\geq 400$. Since the previous statement holds for any choice of
$\chi_i(B\setminus B_{i-1}(a))$ and of activation of edges that  are compatible with $E$, it also holds if 
we only condition on $E$. Since by Claim~\ref{cla:1}, $\bP(E) 
=1-e^{-\Omega(d^{\epsilon/4})}$, we obtain
\begin{align*}
\bP\left(D^i_2(a)\right)\leq \bP(\overline{E})+  \bP\left(D^i_2(a)\mid 
E\right) = e^{-\Omega(d^{\eps/4})}\;.
\end{align*}

\end{proof}

In order to control the condition $(C3)$, for every $b\in B_{i-1}\setminus 
B_i$, we define $D^i_3(b)$ to be the event that $d_{H_i}(b)\leq 
\frac{\delta(\G)}{8\a^2}$.
\begin{lemma}\label{lem:3}
For every $b\in  B_{i-1}\setminus B_i$,
$$
\bP(D^i_3(b))  = e^{-\Omega(d^{\epsilon})} \;.
$$
\end{lemma}
\begin{proof}
If $b\in B_{i-1}\setminus B_i$, then $\Bad_\G(b,\chi_i,H_{i-1})\leq 
\frac{\delta(\G)}{8\a}$; that is, there are at most $ 
\frac{\delta(\G)}{8\a}$ vertices $a\in N_{H_{i-1}}(b)$ with $\chi_i(a)\chi_i(b)\in E(\G)$ such that there exists $b'\in N_{H_{i-1}}(a)\setminus \{b\}$ with $ab'$ activated and $\chi_i(b')=\chi_i(b)$. By $(P3)$ with $c=\chi_i(b)$, there are at least 
$\frac{\delta(\G)}{4\a}$ vertices $a\in N_{H_{i-1}}(b)$ with 
$\chi_i(a)\chi_i(b)\in E(\G)$. Hence, there are at least 
$\frac{\delta(\G)}{8\a}$ vertices $a\in N_{H_{i-1}}(b)$ with 
$\chi_i(a)\chi_i(b)\in E(\G)$ such that either $b$ is the only neighbour of $a$ with 
color $\chi_i(b)$, or if there is another neighbour $b'$ with $\chi_i(b')=\chi_i(b)$, then $ab'$ is not activated and will be deleted in the case that $b'$ retained its color.

We activate every such edge $ab$ independently with probability $1/\a$. Therefore, the 
probability that $d_{H_i}(b)$ is smaller than $k$ is at most the probability that a Binomial random variable with  
$\frac{\delta(\G)}{8\a}$ trials and probability $1/\a$ is smaller than $k$. Since $\delta(\G)\geq d^{\epsilon}$, Chernoff's inequality 
implies that,
$$
\bP(D^i_3(b))= \bP\left(d_{H_i}(b)\leq \frac{\delta(\G)}{16\a^2}\right) = 
e^{-\Omega(d^{\epsilon})} \;.
$$

\end{proof}

In order to control the condition $(C4)$, for every $a\in A$, we define 
$D^1_4(a)$ to be the event that the degree of $a$ to $B\setminus B_1$ in $H_1$ is 
at most $\frac{\delta(\G)}{16\a^2}$, that is $|N_{H_1}(a)\cap (B\setminus 
B_1)|\leq \frac{\delta(\G)}{16\a^2}$.
\begin{lemma}\label{lem:4}
For every $a\in A$,
$$
\bP(D^1_4(a))  = e^{-\Omega(d^{\epsilon})} \;.
$$
\end{lemma}
\begin{proof}
Note that in the first iteration of Phase $II$ we can color $B$ in the following equivalent way: 
\begin{itemize}
 \item[$i)$] \hspace{0.3cm}Color each vertex in $B\setminus N_{H_0}(a)$ independently.
 \item[$ii)$] \hspace{0.3cm}Color each vertex in $N_{H_0}(a)$ independently.
 \item[$iii)$] \hspace{0.3cm}Keep the color classes in $B\setminus N_{H_0}(a)$ but permute the 
colors assigned to each class.
\end{itemize}
After the step $i)$, we say that a color class $C\subseteq B\setminus N_{H_0}(a)$ (with no color assigned, yet) is \emph{dangerous} for $b\in N_{H_0}(a)$ if there are at least  
$\frac{\delta(\G)}{8\a}-1$ activated edges $a'b'$ with $a'\in N_{H_0}(b)\setminus\{a\}$ and $b'\in C$. 

For a $b\in N_{H_0}(a)$, we consider the following set of conditions,
\begin{itemize}
\item[$(F1)$] \hspace{0.3cm} there exists $b'\in N_{H_0}(a)\setminus \{b\}$ with $\chi_1(b')= \chi_1(b)$,
\item[$(F2)$]  \hspace{0.3cm} $\chi_1(b)$ is a dangerous color for $b$,
\item[$(F3)$]  \hspace{0.3cm} $\chi_1(a)\chi_1(b)\notin E(\G)$.
\end{itemize}
Let $Y$ be the number of vertices $b\in N_{H_0}(a)$ that do not satisfy $(F1)$, $(F2)$ and $(F3)$.

Observe that the number of vertices $b\in N_{H_0}(a)$ that do not satisfy $(F1)$ and $(F3)$ is determined by the coloring of $N_{H_0}(a)$ (step $ii)$).
The number of vertices that do not satisfy $(F2)$ is determined by the final coloring of $B\setminus N_{H_0}(a)$ (steps $i)$ and $iii)$) and the activation of the edges between $A$ and $B\setminus N_{H_0}(a)$. We will use the fact these steps are independent to bound $Y$.

Let $Z$ be the number of vertices $b\in N_{H_0}(a)$ that do not satisfy $(F1)$ and $(F3)$. For a given vertex $b$, the probability that $\chi_1(a)\chi_1(b)\in E(\G)$ is at least $p^-=\delta(\G)/\a d$. Given that a color has been assigned to $b$, the probability that this color is unique in $N_{H_0}(a)$ is at least $(1-\frac{1}{\a d})^d\geq 1-1/\a$. Since $d_{H_0}(a)\geq d/4$, we have 
$$
\bE(Z)\geq \left(1-\frac{1}{\a}\right)p^- d_{H_0}(a)\geq \frac{(1-1/\a)\delta(\G)}{4\a}\;.
$$
Similarly as in the proof of Lemma~\ref{lem:many_goodneighs}, we can apply Talagrand's inequality to obtain that 
\begin{align}\label{eq:wakawaka}
\bP\left(Z\leq \frac{\delta(\G)}{8\a}\right)\leq e^{-\Omega(\delta(\G))}\;.
\end{align}

To show that $Y$ is typically large, consider $M$ to be the set of vertices that do not satisfy  $(F1)$ and $(F3)$. Let $b\in M$.
At step $iii)$, each color class will get assigned a random (and different) color and, if in the step $i)$ we have created a small number of dangerous color classes, it will be likely that $b$ does not satisfy $(F2)$.

We define $E$ to be the event that, for every $a'\in N^2_{H_0}(a)$, there are at most $\frac{2d}{\a}$ activated edges $a'b'$ with $b'\notin N_{H_0}(a)$. A simple use of Chernoff's inequality and a union bound, gives $\bP(E) =1-e^{-\Omega(d)}$. The event $E$ is fully determined by the activation of edges between $A$ and $B\setminus N_{H_0}(a)$. Moreover, it allows us to control the number of dangerous color classes.
\begin{claim}\label{cla:3}
If $E$ is satisfied, then for each $b\in N_{H_0}(a)$, the number of 
dangerous color classes for $b$ is at most $\a^{1/2}d$.
\end{claim}
\begin{proof} Since $E$ is satisfied,  for any $a'\in N_{H_0}(b)$, there are at most $2d/\a$ activated edges $a'b'$ with  $b'\notin N_{H_0}(a)$. Since $N_{H_0}(b)$ is rainbow, there are at most $\Delta(\G)\leq \beta 
\delta(\G)$ vertices $a'\in N_{H_0}(b)$ with $\chi_1(a')\chi_1(b)\in E(\G)$ (recall that $\chi_1(b)$ has been fixed at step $ii)$). 
Therefore, there are at most $\frac{2\beta}{\a}\delta(\G) d$ activated edges $a'b'$ 
with $a'\in N_{H_0}(b)$,  $\chi_1(a')\chi_1(b)\in E(\G)$ and $b'\notin N_{H_0}(a)$.  

Recall that a color class $C\subseteq B\setminus N_{H_0}(a)$ is dangerous for $b$ if there are at least  $\frac{\delta(\G)}{8\a}-1\geq \frac{\delta(\G)}{9\a}$ activated edges $a'b'$ with $a'\in N_{H_0}(b)$, $\chi_1(a')\chi_1(b)\in E(\G)$ and $b'\in C$. 

Thus, there are at most $18\beta d \leq \a^{1/2}d$ dangerous color classes, as desired. 
\end{proof}

We now proceed to show that, conditional on $E$, most of the vertices in $M$ do not satisfy $(F2)$. By Claim~\ref{cla:3} and since there are at least $\a d$ colors, the probability that all the dangerous classes for $b$ are assigned in step $iii)$ a color different than $\chi_1(b)$ is at least $1 -\a^{-1/2}$. Thus, 
$$
\bE(Y\mid E,\, M) \geq (1 -\a^{-1/2}) |M|\;.
$$  
Observe that swapping the colors in two color classes of $B\setminus N_{H_0}(a)$ can only change $Y$ by at most $2$, since every vertex in $M$ has a unique color. We can use McDiarmid's inequality to prove that
$$
\bP\left(Y\leq (1-2\alpha^{-1/2})|M| \mid E,\,M\right)\leq e^{-\Omega(|M|)}\;.
$$
Using~\eqref{eq:wakawaka} and $\bP(E)=1- e^{-\Omega(d)}$, we obtain a bound on the unconditional probability,
$$
\bP\left(Y\leq \frac{(1-2\a^{-1/2})\delta(\G)}{8\alpha}\right)\leq e^{-\Omega(d^{\epsilon})}\;,
$$
where we used that $\delta(\G)\geq d^{\epsilon}$.

If $b\in N_{H_0}(a)$ does not satisfy $(F1)$, $(F2)$ and $(F3)$, then $b\in B\setminus B_1$. The edge $ab$ is retained in $H_{1}$ if it was activated in $H_0$. Observe that the  events $(F1)$, $(F2)$ and $(F3)$ are independent of the activation of the edges between $a$ and $N_{H_0}(a)$: $(F1)$ and $(F3)$ do not depend on edge activations, and $(F2)$ only depends on the activation of edges between $A$ and $B\setminus N_{H_0}(a)$. Such an edge $ab$ is activated independently with probability $1/\a$. Using Chernoff's inequality we conclude the proof,
$$
\bP(D^1_4(a))=\bP\left( |N_{H_1}(a)\cap (B\setminus B_1)|\leq 
\frac{\delta(\G)}{16\a^2}\right) = e^{-\Omega(d^{\epsilon})}\;.
$$
\end{proof}

\begin{lemma}\label{lem:5}
Conditions $(C1)$, $(C2)$, $(C3)$ and $(C4)$ hold after the first iteration 
with positive probability.
\end{lemma}
\begin{proof}
By Lemmas~\ref{lem:1},~\ref{lem:2},~\ref{lem:3} and~\ref{lem:4}, we have 
$\bP(D^1_k(a))= e^{-\Omega(d^{\eps/4})}$, for every $k\in\{1,2,3,4\}$. On the 
other hand, $D^1_1(a)$, $D^1_2(a)$, $D^1_3(b)$ and $D^1_4(a)$ are mutually independent 
from every other $D^1_1(a')$, $D^1_2(a')$, $D^1_3(b')$ and $D^1_4(a')$ such that $a'$ 
(or $b'$) is at distance larger than $6$ from $a$ (or $b$) in $H_{0}$. Thus, 
each event is mutually independent from all but at most $4d^6$ other events. 
Since $e^{-\Omega(d^{\eps/4})}=O(d^{-7})$, provided that $d$ is large enough with respect to $\epsilon$, $\beta$ and $\alpha$, we can use the Local Lemma to show that
after the first iteration, $(C1)$, $(C2)$, $(C3)$ and $(C4)$ hold with positive 
probability.
\end{proof}

The same argument suffices to show the following.
\begin{lemma}\label{lem:6}
Let $1\leq i<\tau$. If $(C1)$, $(C2)$ and $(C3)$ hold for every $1\leq j< i$, 
then $(C1)$, $(C2)$ and $(C3)$ hold after the $i$-th iteration with positive 
probability.
\end{lemma}

\begin{proof}[Proof of Lemma~\ref{lem:after_phase_II}.]
Observe that $(Q2)$ and $(Q3)$ are satisfied deterministically, by Phases $I$ 
and $II$ of the coloring procedure. 

We first show that Phase $II$ stops with positive probability. By Lemma~\ref{lem:6} condition $(C2)$ is satisfied with positive probability for every $i\geq 1$. Since the stopping condition of the iterative part of Phase $II$ is $|B_i(a)|\leq d^{\eps/2}$ for every $a\in A$, with positive probability the iterative part of Phase $II$ ends after $O(\log{d})$ iterations.

It remains to show that $(Q1)$ is satisfied with 
positive probability at the end of Phase $II$.
By Lemma~\ref{lem:5}, condition $(C4)$ is satisfied after the first iteration. 
Hence,
$$
d_{H}(a)\geq |N_{H}(a)\cap (B\setminus B_1)| =|N_{H_1}(a)\cap (B\setminus B_1)| 
\geq \frac{\delta(\G)}{16\a^2}\;.
$$

Let us show that the degrees of $b\in B$ in $H$ are also large. First notice 
that the degree of a vertex $b\in B$ only decreases in the iteration when $b$ 
retains its color.

Suppose first that $b\in  B\setminus B_{\tau}$. By Lemma~\ref{lem:6}, condition $(C3)$ is satisfied at every iteration. Therefore,
$$
d_H(b) \geq \frac{\delta(\G)}{16\a^2}\;.
$$
Now suppose that $b\in B_\tau$. Then $b=b_i$, for some $i\geq \tau+1$. Recall 
that Phase $II$ assigns to $b_i$ the color that minimizes 
$\Bad_J(b_i,\chi_{i-1},H_{i-1})$.
First, we show that for every $a\in A$, $d_{H_{i-1}}(a)$ is not too large. By 
Lemma~\ref{lem:6}, $(C2)$ holds at each iteration, which 
implies that the number of uncolored neighbors decreases exponentially fast~(in 
particular, $\tau=O(\log{d})$). Recall that since $\G$ is $\beta$-almost regular, $\Delta(\G)\leq \beta\delta(\G)$.
Since $(C1)$ and $(C2)$ are satisfied at every iteration, we have
\begin{align*}
|N_{H_{i-1}}(a)\cap (B\setminus B_{\tau})| &\leq \sum_{i=1}^{\tau} \GN_i(a) 
\leq \sum_{i=1}^{\tau} \frac{2}{\a} p^+|B_i(a)| + O(d^{\eps/2}\log{d}) 
\nonumber\\
&\leq \left(\sum_{i=1}^{\infty} 
\frac{1}{\a^{i/2}}\right)\frac{2\Delta(\G)}{\a^2 d}|B_0(a)| + 
O(d^{\eps/2}\log{d}) \nonumber\\
&\leq \frac{4\Delta(\G)}{\a^2}\leq \frac{\delta(\G)}{2\a}\;,
\end{align*}
provided that $d$ is large enough and since $\alpha\geq 25\beta^2$. Here we also used that 
$|B_0(a)|\leq d$.
By the definition of $\tau$, we also have $|N_{H_{i-1}}(a)\cap B_{\tau}|= |B_{\tau}(a)| \leq d^{\eps/2}$. 
This implies that, for any $i>\tau$
\begin{align}\label{eq:5}
d_{H_{i-1}}(a)&\leq \frac{\delta(\G)}{2\a}+ d^{\eps/2}\leq 
\frac{\delta(\G)}{\a}\;.
\end{align}
Now we can lower bound the degree of $b_i$ in $H$. By~(\ref{eq:5}), 
there are at most $\frac{\delta(\G)d}{\a}$ edges in $H_{i-1}$ incident to a neighbour of $b_i$.
Thus, there is a color $c\in V(\G)$ such that if $\chi_i(b_i)=c$, 
$\Bad_{\G}(b_i,\chi_{i},H_{i-1})$ is small, in fact
$$
|\{a\in N_{H_{i-1}}(b_i):\;\exists b'\in 
N_{H_{i-1}}(b_i)\setminus\{b_i\}\mbox{ and } \chi_{i}(b')=c  \}|\leq 
\frac{\delta(\G)}{\a^{2}}\;.
$$
We set $\chi_i(b_i)=c$ and delete all the edges $ab_i$ for some $a\in A$ that either 
may cause conflicts (at most $\delta(\G)/\a^2$) or $\chi_i(a)\chi_i(b_i)\notin E(\G)$. Since property $(P3)$ holds for every $b\in B$ and every color 
$c\in V(\G)$, the degree of $b_i$ in $H$ is
$$
d_{H}(b_i)\geq \frac{\delta(\G)}{4\a} - \frac{\delta(\G)}{\a^{2}}\geq  
\frac{\delta(\G)}{16\a^2}\;,
$$
provided that $\a\geq 5$.

\end{proof}

We conclude with the proof of the main theorem.
\begin{proof}[Proof of Theorem~\ref{thm:more_general}.]
Let $H$ be the spanning subgraph and let $\chi$ be the partial coloring of $H$ provided by Lemma~\ref{lem:after_phase_II}. For the sake of contradiction, suppose that $H$ contains a locally injective copy of  $F\in \F$. Consider the coloring of $F$ induced by the colors given by $\chi$ to the locally injective copy of $F$ in $H$. Since for every edge $ab$ in $H$ we have $\chi(a)\chi(b)\in E(\G)$, the coloring of $F$ induces an homomorphism from $F$ to $\G$.  Moreover, since for every vertex $v\in V(H)$, $N_H(v)$ is rainbow, this homomorphism is locally injective. However, $\hom^*(F,\G)=0$ by the hypothesis of the theorem and we obtain a contradiction. Thus, $\hom^*(F,H)=0$. Finally, the subgraph $H$ also satisfies
$\delta(H)=\Omega(\delta(\G))$.
\end{proof}


\providecommand{\bysame}{\leavevmode\hbox to3em{\hrulefill}\thinspace}
\providecommand{\MR}{\relax\ifhmode\unskip\space\fi MR }
\providecommand{\MRhref}[2]{%
  \href{http://www.ams.org/mathscinet-getitem?mr=#1}{#2}
}
\providecommand{\href}[2]{#2}
\begin{thebibliography}{}

\end{thebibliography}


\begin{thebibliography}{10}

\bibitem{ars1999}
N.~Alon, L.~R{\'o}nyai, and T.~Szab{\'o}, \emph{Norm-graphs: variations and
  applications}, Journal of Combinatorial Theory, Series B \textbf{76} (1999),
  no.~2, 280--290.

\bibitem{b1966}
W.~G. Brown, \emph{On graphs that do not contain a {T}homsen graph}, Canadian
  Mathematical Bulletin \textbf{9} (1966), no.~2, 1--2.

\bibitem{cfs2014}
D.~Conlon, J.~Fox, and B.~Sudakov, \emph{Large subgraphs without complete
  bipartite graphs}, arXiv:1401.6711 (2014).

\bibitem{cfs2015}
D.~Conlon, J.~Fox, and B.~Sudakov, \emph{Short proofs of some extremal results {II}}, to appear in J.
  Combin. Theory Ser. B. (2015).

\bibitem{df2014}
M.~Delcourt and A.~Ferber, \emph{On a {C}onjecture of {T}homassen}, Electronic
  Journal of Combinatorics \textbf{22} (2015), no.~3, P3.2.

\bibitem{e1959}
P.~Erd{\H{o}}s, \emph{Graph theory and probability {I}}, Canadian Journal of
  Mathematics \textbf{11} (1959), 34--38.

\bibitem{e1961}
P.~Erd{\H{o}}s, \emph{Graph theory and probability {II}}, Canadian Journal of
  Mathematics \textbf{13} (1961), 346--352.

\bibitem{e1968}
P.~Erd{\H{o}}s, \emph{Problem 1, {T}heory of {G}raphs}, Proc. Colloq., Tihany
  1966 (1968), 361--362.

\bibitem{ers1966}
P.~Erd{\H{o}}s, A.~R{\'e}nyi, and V.~T. S{\'o}s, \emph{On a problem of graph
  theory}, Studia Scientiarum Mathematicarum Hungarica \textbf{1} (1966),
  215--235.

\bibitem{es1970}
P.~Erd{\H{o}}s and M.~Simonovits, \emph{Some extremal problems in graph
  theory}, Combinatorial theory and its applications, {I} ({P}roc. {C}olloq.,
  {B}alatonf\"ured, 1969), North-Holland, Amsterdam, 1970, pp.~377--390.

\bibitem{fkp2014}
F.~Foucaud, M.~Krivelevich, and G.~Perarnau, \emph{Large subgraphs without
  short cycles}, SIAM Journal of Discrete Mathematics \textbf{29} (2015),
  no.~1, 65--78.

\bibitem{fs2013}
Z.~F{\"u}redi and M.~Simonovits, \emph{The history of degenerate (bipartite)
  extremal graph problems}, Erd{\H{o}}s Centennial, Springer, 2013,
  pp.~169--264.

\bibitem{krs1996}
J{\'a}nos Koll{\'a}r, Lajos R{\'o}nyai, and Tibor Szab{\'o}, \emph{Norm-graphs
  and bipartite tur{\'a}n numbers}, Combinatorica \textbf{16} (1996), no.~3,
  399--406.

\bibitem{kst1954}
T.~K{\H{o}}v{\'a}ri, V.~T. S{\'o}s, and P.~Tur{\'a}n, \emph{On a problem of {K.
  Zarankiewicz}}, Colloquium Mathematicae \textbf{3} (1954), no.~1, 50--57.

\bibitem{ko2004}
D.~K{\"u}hn and D.~Osthus, \emph{Every graph of sufficiently large average
  degree contains a ${C}_4$-free subgraph of large average degree},
  Combinatorica \textbf{24} (2004), no.~1, 155--162.

\bibitem{k2013}
G.~Kun, \emph{Expanders have a spanning {L}ipschitz subgraph with large girth},
  arXiv:1303.4982 (2013).

\bibitem{lv2005}
T.~Lam and J.~Verstra{\"e}te, \emph{A note on graphs without short even
  cycles}, Electronic Journal of Combinatorics \textbf{12} (2005), no.~1, Paper
  5.

\bibitem{luw1999}
F.~Lazebnik, V.~A. Ustimenko, and A.~J. Woldar, \emph{Polarities and
  $2k$-cycle-free graphs}, Discrete Mathematics \textbf{197} (1999), 503--513.

\bibitem{mr2002}
M.~Molloy and B.~Reed, \emph{Graph colouring and the probabilistic method},
  vol.~23, Springer, 2002.

\bibitem{t1983}
C.~Thomassen, \emph{Girth in graphs}, Journal of Combinatorial Theory Series B
  \textbf{35} (1983), no.~2, 129--141.

\bibitem{t1989}
C.~Thomassen, \emph{Configurations in graphs of large minimum degree, connectivity,
  or chromatic number}, Annals of the New York Academy of Sciences \textbf{555}
  (1989), no.~1, 402--412.

\end{thebibliography}

\end{document}